\documentclass[10pt]{amsart}
\usepackage[leqno]{amsmath}
\usepackage{amssymb,latexsym,soul,cite,amsthm,color,enumitem,graphicx,mathtools,microtype,accents}
\usepackage[colorlinks=true,urlcolor=ferngreen,citecolor=ferngreen,linkcolor=ferngreen,linktocpage,pdfpagelabels,bookmarksnumbered,bookmarksopen]{hyperref}
\definecolor{ferngreen}{rgb}{0.31, 0.47, 0.26}
\usepackage[english]{babel}
\usepackage[left=2.5cm,right=2.5cm,top=2.5cm,bottom=2.5cm]{geometry}
\usepackage[T1]{fontenc}

\numberwithin{equation}{section}

\newtheorem{theorem}{Theorem}[section]
\theoremstyle{plain}
\newtheorem{lemma}[theorem]{Lemma}
\theoremstyle{plain}
\newtheorem{proposition}[theorem]{Proposition}
\theoremstyle{plain}

\theoremstyle{definition}
\newtheorem{remark}[theorem]{Remark}
\newtheorem{example}[theorem]{Example}

\newcommand{\N}{{\mathbb N}}

\newcommand{\R}{{\mathbb R}}
\newcommand{\eps}{\varepsilon}
\newcommand{\beq}{\begin{equation}}
\newcommand{\eeq}{\end{equation}}
\renewcommand{\le}{\leqslant}
\renewcommand{\ge}{\geqslant}

\newcommand{\w}{W^{s,p}_0(\Omega)}
\newcommand{\fpl}{(-\Delta)_p^s\,}
\newcommand{\ds}{{\rm d}_\Omega^s}

\makeatletter
\newcommand{\leqnomode}{\tagsleft@true}
\newcommand{\reqnomode}{\tagsleft@false}
\makeatother

\newenvironment{enumroman}{\begin{enumerate}

}{\end{enumerate}}

\title[Optimal solvability for the fractional $p$-Laplacian]{Optimal solvability for the fractional $p$-Laplacian\\ with Dirichlet conditions}

\author[A.\ Iannizzotto, D.\ Mugnai]{Antonio Iannizzotto, Dimitri Mugnai}

\address[A.\ Iannizzotto]{Dipartimento di Matematica e Informatica
\newline\indent
Universit\`a degli Studi di Cagliari
\newline\indent
Via Ospedale 72, 09124 Cagliari, Italy}
\email{antonio.iannizzotto@unica.it}

\address[D.\ Mugnai]{Dipartimento di Scienze Ecologiche e Biologiche
\newline\indent
Universit\`a degli Studi della Tuscia
\newline\indent
Largo dell'Universit\`a, 01100 Viterbo, Italy}
\email{dimitri.mugnai@unitus.it}


\begin{document}

\begin{abstract}
We study a nonlinear, nonlocal Dirichlet problem driven by the fractional $p$-Laplacian, involving a $(p-1)$-sublinear reaction. By means of a weak comparison principle we prove uniqueness of the solution. Also, comparing the problem to 'asymptotic' weighted eigenvalue problems for the same operator, we prove a necessary and sufficient condition for the existence of a solution. Our work extends classical results due to Brezis-Oswald \cite{BO} and Diaz-Saa \cite{DS} to the nonlinear nonlocal framework.
\end{abstract}

\maketitle

\begin{center}
Version of \today\
\end{center}

\section{Introduction}\label{sec1}

The present paper is devoted to the study of the following Dirichlet boundary value problem:
\beq\label{dir}
\begin{cases}
\fpl u = f(x,u) & \text{in $\Omega$} \\
u > 0 & \text{in $\Omega$} \\
u = 0 & \text{in $\Omega^c$.}
\end{cases}
\eeq
Here $\Omega\subset\R^N$ ($N\ge 2$) is a bounded domain with a $C^{1,1}$-boundary $\partial\Omega$, $p>1$, $s\in(0,1)$ are real numbers, and the leading operator is the fractional $p$-Laplacian, defined for a sufficiently smooth function $u:\R^N\to\R$ by
\[\fpl u(x) = 2\lim_{\eps\to 0^+}\int_{B^c_\eps(x)}\frac{|u(x)-u(y)|^{p-2}(u(x)-u(y))}{|x-y|^{N+ps}}\,dy,
\]
where $B_\eps(x)$ denotes the ball centered at $x\in\R^N$ with radius $\eps>0$. This is a nonlinear, nonlocal operator which in special cases reduces, up to a multiplicative constant, to the fractional Laplacian ($p=2$, $s\in(0,1)$), to the $p$-Laplacian ($p>1$, $s=1$), and in particular to the Laplacian ($p=2$, $s=1$). The fractional $p$-Laplacian is degenerate if $p>2$, singular if $p\in(1,2)$. An introduction to this operator and related problems can be found in \cite{P}.
\vskip2pt

Our hypotheses on the reaction $f$ are the following:
\begin{itemize}[leftmargin=1cm]
\item[${\bf H}$] $f:\Omega\times\R^+\to\R$ is a Carath\'eodory function s.t.\
\begin{enumerate}
\item\label{h1} $f(\cdot,t)\in L^\infty(\Omega)$ for all $t\in\R^+$;
\item\label{h2} there exists $c_0>0$ s.t.\ for a.e.\ $x\in\Omega$ and all $t\in\R^+$
\[f(x,t) \le c_0(1+t^{p-1});\]
\item\label{h3} for a.e.\ $x\in\Omega$ the mapping
\[t \mapsto \frac{f(x,t)}{t^{p-1}}\]
is strictly decreasing in $\R^+_0$.
\end{enumerate}
\end{itemize}

\begin{remark}\label{hcomm}
Some comments on hypotheses ${\bf H}$ are in order:
\begin{itemize}[leftmargin=1cm]
\item[$(a)$] The boundedness condition ${\bf H}$ (\ref{h1}) is obviously satisfied in the autonomous case, i.e., $f\in C(\R^+)$.
\item[$(b)$] The growth condition ${\bf H}$ (\ref{h2}) acts on the reaction {\it from above only}, as is the case in \cite{BO,DS} but differently from previous results in the nonlocal setting such as \cite{BMV,BMV2,MPV}.
\item[$(c)$] The strict monotonicity condition ${\bf H}$ (\ref{h3}) classifies our reaction as a $(p-1)$-sublinear one.
\end{itemize}
\end{remark}

\begin{example}
We present here two examples of autonomous reactions satisfying ${\bf H}$. First, for all $1<q\le p<r$, we recall the sub- and equidiffusive logistic reactions
\[f(t) = t^{q-1}-t^{r-1}.\]
Also, for any $\alpha\geq p-1$, $r>p$ we have the exponentially growing reaction
\[f(t) = \begin{cases}
t^{p-1}-t^{r-1} & \text{if $t\in[0,1]$} \\
t^{p-1}-e^{t^\alpha-1} & \text{if $t>1$},
\end{cases}\]
Note that both may have a super-critical growth from below (see \cite{AS} for other results related to supercritical fractional $p$-Laplacian equations).
\end{example}

The study of boundary value problems with sublinear reactions dates back to the classical work of Brezis and Oswald \cite{BO}, dealing with the Laplacian as leading operator ($p=2$, $s=1$) and Dirichlet boundary conditions, with hypotheses analogous to ${\bf H}$. In \cite{BO} the authors prove that the problem admits at most one solution, and provide a characteristic condition for the existence of such a solution (this is called in the current literature ``optimal solvability"). In the following years, similar results have been proved for a variety of nonlinear local elliptic operators, such as the $p$-Laplacian with Dirichlet conditions \cite{DS} or Neumann conditions \cite{GP}, or a general nonlinear operator with Robin conditions \cite{FMP}. See also \cite{M} for an alternative approach based on nonsmooth critical point theory.
\vskip2pt

In the nonlocal framework, we recall the related results for the fractional Laplacian in $\R^N$ \cite{PT}, for the spectral fractional Laplacian in a bounded domain \cite{MBR}, for the fractional $p$-Laplacian with nonlocal Robin conditions \cite{MPV}, and a mixed local-nonlocal operator with Dirichlet conditions \cite{BMV,BMV2}. We remark again that the last three contributions present partial results regarding the necessary condition for existence, and they both employ bilateral growth conditions on the reaction. A different type of optimal solvability result for the fractional $p$-Laplacian with Dirichlet conditions and a critical reaction was obtained in \cite{BS}.
\vskip2pt

Our result is the first exact counterpart of \cite{BO,DS} for the Dirichlet fractional $p$-Laplacian, and to our knowledge it is new even for the linear case, i.e.\ for the Dirichlet fractional Laplacian ($p=2$, $s\in(0,1)$). We relate the solvability of problem \eqref{dir} to the signs of two weighted 'eigenvalues', defined as follows. By ${\bf H}$  (\ref{h3}), for a.e.\ $x\in\Omega$ we may define
\[a_0(x) = \lim_{t\to 0^+}\frac{f(x,t)}{t^{p-1}}, \ a_\infty(x) = \lim_{t\to\infty}\frac{f(x,t)}{t^{p-1}}.\]
We have for a.e.\ $x\in\Omega$
\[a_0(x) \ge f(x,1) \ge a_\infty(x),\]
so by ${\bf H}$  (\ref{h1}) we can find $C>0$ s.t.\ $a_0\ge -C$, $a_\infty\le C$ in $\Omega$. On the contrary, $a_0=+\infty$ and $a_\infty=-\infty$ may occur on non-null subsets of $\Omega$. For any measurable function $a$ defined in $\Omega$, possibly taking one of the values $\pm\infty$ (but not both), we set
\beq\label{la1}
\lambda_1(a) = \inf_{v\in\w\setminus\{0\}}\frac{\|v\|^p-\int_{\{v\neq 0\}}a(x)|v|^p\,dx}{\|v\|_p^p} \in \R\cup\{\pm\infty\}
\eeq
(see Section \ref{sec2} for the notation). If $a\in L^q(\Omega)$ (for convenient $q>1$), then $\lambda_1(a)\in\R$ is the principal eigenvalue of the following weighted eigenvalue problem:
\beq\label{evp}
\begin{cases}
\fpl v-a(x)|v|^{p-2}v = \lambda|v|^{p-2}v & \text{in $\Omega$} \\
v = 0 & \text{in $\Omega^c$}
\end{cases}
\eeq
For a discussion on weighted and nonweighted eigenvalues of $\fpl$, see \cite{BP,I,LL,MBSZ}. Using eigenvalues as asymptotic thresholds for general nonlinear reactions is certainly not new in the study of elliptic problems driven by the fractional $p$-Laplacian, see for instance \cite{FI,ILPS,IL,IMP} (the last one dealing with the logistic equation, with a result agreeing with those of the present paper in the sub- and equidiffusive case).
\vskip2pt

Our main result is the following:

\begin{theorem}\label{main}
Let $\Omega\subset\R^N$ ($N\ge 2$) be a bounded domain with a $C^{1,1}$-boundary $\partial\Omega$, $p>1$, $s\in(0,1)$ be real numbers, $f$ satisfy ${\bf H}$. Then, problem \eqref{dir}
\begin{enumroman}
\item\label{main1} has at most one solution;
\item\label{main2} has a solution iff $\lambda_1(a_0)<0<\lambda_1(a_\infty)$.
\end{enumroman}
\end{theorem}

We note that, for an autonomous reaction, the inequality in \ref{main2} is equivalent to the following, where $\lambda_1>0$ denotes the principal eigenvalue of \eqref{evp} with $a=0$:
\[\lim_{t\to\infty}\frac{f(t)}{t^{p-1}} < \lambda_1 < \lim_{t\to 0^+}\frac{f(t)}{t^{p-1}}.\]
For a definition and a discussion on the notion of 'solution', as well as some basic properties of $\fpl$, we refer to Section \ref{sec2}. In Section \ref{sec3} we shall prove the uniqueness statement \ref{main1}. Regarding the existence statement \ref{main2}, we will first prove in Section \ref{sec4} that the condition is necessary by exploiting the properties of the eigenvalue problem \eqref{evp} proved in \cite{I,MBSZ}. Finally, in Section \ref{sec5} we will tackle the more delicate issue of the sufficient condition for existence, by using a variational approach and introducing a sequence of auxiliary truncated problems.
\vskip2pt

We essentially follow the approach of the original papers \cite{BO,DS}, but with some important differences typical of the nonlocal framework, which deserve to be laid out:
\begin{itemize}[leftmargin=1cm]
\item[$(a)$] In \cite{DS} a form of 'hidden convexity' (i.e., the energy functional is convex in the variable $u^\frac{1}{p}$) represents a useful tool for uniqueness. Though an equivalent form of convexity for the fractional energy was proved in \cite{FP} (see also \cite{BF}), we prefer to follow a different (and in our opinion simpler) approach, based on a discrete Picone's inequality from \cite{BS} and a comparison argument from \cite{IL}
\item[$(b)$] Although enjoying good interior regularity, solutions of fractional order equations are generally singular at the boundary, which prevents a classical Hopf's boundary lemma from holding, thus making it difficult to work with quotients between solutions (as extensively done in \cite{BO,DS}). In fact, as proved in \cite{IMS1}, global regularity of a solution $u$ only amounts at $u\in C^\alpha(\overline\Omega)$ for some $\alpha\in(0,s]$. Nevertheless, some useful boundary estimates on $u$ in terms of the distance function $\ds(x)={\rm dist}(x,\Omega^c)^s$ and a related fractional Hopf's lemma from \cite{IMP} can be employed, along with the technical Lemma \ref{fra}, to overcome such difficulty (in the linear and degenerate cases $p\ge 2$ a better regularity theory holds, see Remark \ref{deg}).
\item[$(c)$] Due to the nonlocal nature of the operator, in the proof of the sufficient condition it is not immediate to compare the energies of solutions of \eqref{dir} and the truncated problems, respectively. We deal with such issue by applying a special submodularity inequality from \cite{GM}.
\end{itemize}
\vskip4pt
\noindent
{\bf Notations.} Throughout the paper we shall use the following notations:
\begin{itemize}[leftmargin=1cm]
\item $\R^+=[0,\infty)$, $\R^-=(-\infty,0]$, and $\R^+_0=(0,\infty)$.
\item $A^c=\R^N\setminus A$ for all $A\subset\R^N$.
\item $f\le g$ in $\Omega$ means that $f(x)\le g(x)$ for a.e.\ $x\in\Omega$ (and similar expressions), for any two measurable functions $f,g:\Omega\to\R$. 
\item $f\vee g=\max\{f,g\}$, $f\wedge g=\min\{f,g\}$, and $f^\pm=(\pm f)\vee 0$ are the positive and negative parts of $f$, respectively.
\item $X_+$ is the positive order cone of an ordered Banach space $X$.
\item $\|\cdot\|_q$ denotes the standard norm of $L^q(\Omega)$ (or $L^q(\R^N)$, which will be clear from the context), for any $q\in[1,\infty]$.
\item $\|\cdot\|$ is the reference norm defined in Section \ref{sec2}.
\item Every function $u$ defined in $\Omega$ is identified with its $0$-extension to $\R^N$.
\item$j_p(t)=|t|^{p-2}t$ for all $t\in\R$.
\item $C$ will denote positive universal constants whose value may change case by case.
\end{itemize}

\section{Preliminaries}\label{sec2}

In this section we fix a functional-analytic framework for problem \eqref{dir} and recall some technical results, referring the reader to \cite{DPV} for details. From now on, $\Omega$, $p$, and $s$ will be as in Section \ref{sec1}. First, for any measurable $u:\R^N\to\R$ we define the Gagliardo seminorm
\[[u]_{s,p} = \Big(\iint_{\R^N\times\R^N}\frac{|u(x)-u(y)|^p}{|x-y|^{N+ps}}\,dx\,dy\Big)^\frac{1}{p}.\]
Then, we define the fractional Sobolev spaces
\[W^{s,p}(\R^N) = \big\{u\in L^p(\R^N):\,[u]_{s,p}<\infty\big\},\]
\[\w = \big\{u\in W^{s,p}(\R^N):\,u=0 \ \text{in $\Omega^c$}\big\}.\]
The latter is a uniformly convex, separable Banach space under the norm $\|u\|=[u]_{s,p}$, with dual space $W^{-s,p'}(\Omega)$. The embedding $\w\hookrightarrow L^q(\Omega)$ is compact for all $q\in[1,p^*_s)$, where
\[p^*_s = \begin{cases}
\frac{Np}{N-ps} & \text{if $ps<N$} \\
\infty & \text{if $ps\ge N$.}
\end{cases}\]
The slight abuse of notation in the definition of $\w$ is justified by the fact that, if $\partial\Omega$ is smooth enough, the above-defined space coincides with the completion of $C^\infty_c(\Omega)$ with respect to the Gagliardo seminorm defined in $\Omega\times\Omega$ (see for instance \cite{BP}). Also, $\w$ has a lattice structure and the following submodularity inequality holds for all $u,v\in\w$ (see \cite[Remark 3.3]{GM}):
\beq\label{sub}
\|u\vee v\|^p+\|u\wedge v\|^p \le \|u\|^p+\|v\|^p.
\eeq
We define the fractional $p$-Laplacian as an operator $\fpl:\w\to W^{-s,p'}(\Omega)$ s.t.\ for all $u,\varphi\in\w$
\[\langle\fpl u,\varphi\rangle = \iint_{\R^N\times\R^N}\frac{j_p(u(x)-u(y))(\varphi(x)-\varphi(y))}{|x-y|^{N+ps}}\,dx\,dy.\]
Equivalently, $\fpl$ is the gradient of the functional
\[\w \ni u \mapsto \frac{\|u\|^p}{p}.\]
Such a definition agrees with the one given in Section \ref{sec1} if $u$ is smooth enough (see \cite[Proposition 2.12]{IMS1}). It can be seen that $\fpl$ is a continuous, $(S)_+$-operator, i.e., whenever $u_n\rightharpoonup u$ in $\w$ and
\[\limsup_n\langle\fpl u_n,u_n-u\rangle \le 0,\]
then $u_n\to u$ in $\w$ (see for instance \cite[Lemma 2.1]{FI}). Plus, for all $u\in\w$ we have
\[\|u^\pm\|^p \le \langle\fpl u,\pm u^\pm\rangle.\]
Now we focus on problem \eqref{dir}. First, we extend the reaction $f$ by setting for all $(x,t)\in\Omega\times\R^-_0$
\[f(x,t) = f(x,0).\]
We say that $u\in\w$ is a (weak) solution of \eqref{dir}, if for all $\varphi\in\w$
\beq\label{wsl}
\langle\fpl u,\varphi\rangle = \int_\Omega f(x,u)\varphi\,dx.
\eeq
Since ${\bf H}$ does not include a full growth condition on $f(x,\cdot)$, the right-hand side of \eqref{wsl} may {\it a priori} blow up for some $u\in\w$. We need some labor to show that \eqref{wsl} is well posed and can be reasonably assumed as a weak formulation of \eqref{dir}. We begin with a weak minimum principle:

\begin{lemma}\label{wmp}
Let ${\bf H}$ hold, $u\in\w$ satisfy \eqref{wsl}. Then, $u\ge 0$ in $\Omega$.
\end{lemma}
\begin{proof}
First, for all $\delta>0$ we can find $C_\delta>0$ s.t.\ for a.e.\ $x\in\Omega$ and all $t\in[0,\delta]$
\beq\label{wmp1}
f(x,t) \ge -C_\delta t^{p-1}.
\eeq
Indeed, by ${\bf H}$  (\ref{h1})  (\ref{h3}) we have
\[f(x,t) \ge \frac{f(x,\delta)}{\delta^{p-1}}t^{p-1} \ge -\frac{\|f(\cdot,\delta)\|_\infty}{\delta^{p-1}}t^{p-1},
\]
as claimed. Moreover, by letting $t\to 0^+$ in \eqref{wmp1} and recalling the extended definition of $f$, we have for a.e.\ $x\in\Omega$ and all $t\in\R^-$
\beq\label{wmp2}
f(x,t) \ge 0.
\eeq
Now test \eqref{wsl} with $-u^-\in\w$ and use \eqref{wmp2}:
\[\|u^-\|^p \le \langle\fpl u,-u^-\rangle = \int_{\{u<0\}}f(x,u)u\,dx \le 0.\]
So, $u^-=0$, which concludes the proof.
\end{proof}

As pointed out in Section \ref{sec1}, an important role in nonlocal problem is played by the distance function defined for all $x\in\R^N$ by
\[{\rm d}_\Omega(x) = {\rm dist}(x,\Omega^c).\]
Indeed, we have the following regularity result and boundary estimates:

\begin{lemma}\label{reg}
Let ${\bf H}$ hold, $u\in\w$ satisfy \eqref{wsl}. Then,
\begin{enumroman}
\item\label{reg1} $u\in C^\alpha(\overline\Omega)$ for some $\alpha\in(0,s]$;
\item\label{reg2} there exists $C>0$ s.t.\ in $\Omega$
\[0 \le u \le C\ds.\]
\end{enumroman}
\end{lemma}
\begin{proof}
We assume $ps<N$, the remaining cases being easily solved by fractional Sobolev embeddings (see \cite[Theorems 6.9, 6.10]{DPV}). By Lemma \ref{wmp} we have $u\ge 0$ in $\Omega$. So, arguing as in \cite[Theorem 3.3]{CMS} and using only the growth condition from above ${\bf H}$  (\ref{h2}), we see that $u\in L^q(\Omega)$ for all $q>1$. Then fix
\[r > \max\Big\{\frac{N}{ps},\,\frac{1}{p-1}\Big\},\]
so that $u^{p-1}\in L^r(\Omega)$. By ${\bf H}$  (\ref{h2}) again and arguing as in \cite[Theorem 3.1]{BP}, we get that $u\in L^\infty(\Omega)$. Now ${\bf H}$  (\ref{h2}) and \eqref{wmp1} (with $\delta=\|u\|_\infty$) imply
\[f(\cdot,u) \in L^\infty(\Omega).\]
Thus, by \cite[Theorem 1.1]{IMS1} we have \ref{reg1} and by \cite[Theorem 4.4]{IMS1} we have \ref{reg2}.
\end{proof}

Next, we improve Lemma \ref{wmp} to a strong minimum principle, incorporating a fractional Hopf-type property (see \cite{DQ} for a similar result in the pure power case):

\begin{lemma}\label{smp}
Let ${\bf H}$ hold, $u\in\w\setminus\{0\}$ satisfy \eqref{wsl}. Then,
\[\inf_\Omega\frac{u}{\ds} > 0.\]
\end{lemma}
\begin{proof}
By Lemma \ref{reg} we have $u\in C^\alpha(\overline\Omega)_+\setminus\{0\}$. By ${\bf H}$  (\ref{h1}) we may set for all $t\in\R$
\[g(t) = \frac{\|f(\cdot,\|u\|_\infty)\|_\infty}{\|u\|_\infty^{p-1}}j_p(t).\]
Clearly $g\in C(\R)\cap BV_{\rm loc}(\R)$ (being monotone). By \eqref{wsl} and arguing as in Lemma \ref{wmp}, we have weakly in $\Omega$
\[\fpl u+g(u) \ge 0 = g(0).\]
By \cite[Theorem 2.6]{IMP} we conclude.
\end{proof}

In particular, by Lemma \ref{smp} we have $u>0$ in $\Omega$. In conclusion, we see that \eqref{wsl} is well posed and that any function $u\neq 0$ satisfying \eqref{wsl} actually solves \eqref{dir}.
\vskip2pt

The following technical result, which in a sense simplifies some arguments in \cite{BS,IL,MPV}, will be very useful in the next sections (recall that we tacitly identify functions defined in $\Omega$ with their $0$-extensions to $\R^N$):

\begin{lemma}\label{fra}
Let $u,v\in\w\cap C(\overline\Omega)$, $C>1$ be s.t.\ in $\Omega$
\[\frac{\ds}{C} \le u,v \le C\ds.\]
Then,
\[\frac{u^p}{v^{p-1}}\in\w.\]
\end{lemma}
\begin{proof}
First we note that in $\Omega$
\begin{equation}\label{rapporto}
0 < \frac{u}{v} = \frac{u}{\ds}\frac{\ds}{v} \le C^2,
\end{equation}
so $u/v\in L^\infty(\Omega)$. Since $u\in L^\infty(\Omega)$, we immediately see that also
\[\frac{u^p}{v^{p-1}} \in 
L^\infty(\Omega),
\]
while the same function vanishes in $\Omega^c$. There remains to show that the Gagliardo seminorm of $u^p/v^{p-1}$ is finite. To do that, we first show that there exists $C>0$ s.t.\ for all $x,y\in\R^N$
\beq\label{fra1}
\Big|\frac{u(x)^p}{v(x)^{p-1}}-\frac{u(y)^p}{v(y)^{p-1}}\Big| \le C|u(x)-u(y)|+C|v(x)-v(y)|.
\eeq
First, by symmetry and by recalling that all involved functions vanish in $\Omega^c$, we may assume $x,y\in\Omega$ and $u(x)\ge u(y)$. Indeed, if $x\in \Omega$ and $y\not\in \Omega$, by \eqref{rapporto} we have
\[\Big|\frac{u(x)^p}{v(x)^{p-1}}\Big| \le Cu(x)\]
So, if $x,y\in \Omega$, by monotonicity of the maps $t\mapsto t^{p-1},\,t^{p-2}$ in $\R^+_0$ and Lagrange's rule we have
\begin{align}\label{fra2}
\Big|\frac{u(x)^p}{v(x)^{p-1}}-\frac{u(y)^p}{v(y)^{p-1}}\Big| &\le \frac{|u(x)^p-u(y)^p|}{v(x)^{p-1}}+u(y)^p\frac{|v(x)^{p-1}-v(y)^{p-1}|}{v(x)^{p-1}v(y)^{p-1}} \\
\nonumber &\le p\frac{u(x)^{p-1}}{v(x)^{p-1}}|u(x)-u(y)|+(p-1)u(y)^p\frac{(v(x)^{p-2})\vee(v(y)^{p-2})}{v(x)^{p-1}v(y)^{p-1}}|v(x)-v(y)|.
\end{align}
The first term is easily estimated by recalling that $u/v\in L^\infty(\Omega)$. For the second term, we distinguish two cases:
\begin{itemize}[leftmargin=1cm]
\item[$(a)$] if $v(x)^{p-2}\ge v(y)^{p-2}$, then we have
\[u(y)^p\frac{(v(x)^{p-2})\vee(v(y)^{p-2})}{v(x)^{p-1}v(y)^{p-1}} \le \frac{u(y)^p}{v(x)v(y)^{p-1}} \le \frac{u(x)}{v(x)}\,\frac{u(y)^{p-1}}{v(y)^{p-1}} \le \Big\|\frac{u}{v}\Big\|_\infty^p;\]
\item[$(b)$] if $v(x)^{p-2}< v(y)^{p-2}$, then we have
\[u(y)^p\frac{(v(x)^{p-2})\vee(v(y)^{p-2})}{v(x)^{p-1}v(y)^{p-1}} \le \frac{u(y)^p}{v(x)^{p-1}v(y)} \le \frac{u(x)^{p-1}}{v(x)^{p-1}}\,\frac{u(y)}{v(y)} \le \Big\|\frac{u}{v}\Big\|_\infty^p.\]
\end{itemize}
Plugging the above estimates into \eqref{fra2}, we have
\[\Big|\frac{u(x)^p}{v(x)^{p-1}}-\frac{u(y)^p}{v(y)^{p-1}}\Big| \le p \Big\|\frac{u}{v}\Big\|_\infty^{p-1}|u(x)-u(y)|+(p-1) \Big\|\frac{u}{v}\Big\|_\infty^p|v(x)-v(y)|,\]
so \eqref{fra1} is proved. Now, integrating \eqref{fra1} and using the elementary inequality
\[(a+b)^p \le 2^{p-1}(a^p+b^b) \ \text{for all $a,b\in\R^+$,}\]
we get
\[\iint_{\R^N\times\R^N}\Big|\frac{u(x)^p}{v(x)^{p-1}}-\frac{u(y)^p}{v(y)^{p-1}}\Big|^p\,\frac{dx\,dy}{|x-y|^{N+ps}} \le C\|u\|^p+C\|v\|^p < \infty.\]
So we conclude that $u^p/v^{p-1}\in\w$.
\end{proof}

\begin{remark}\label{deg}
In the degenerate and linear cases $p\ge 2$, much more can be said about the regularity of weak solutions. First, regarding pure H\"older regularity, we have $u\in C^s(\overline\Omega)$. More important, let us introduce the weighted H\"older spaces
\[C^\alpha_s(\overline\Omega) = \Big\{u\in C^0(\overline\Omega):\,\frac{u}{\ds} \ \text{has a $\alpha$-H\"older continuous extension to $\overline\Omega$}\Big\}\]
(for $\alpha=0$, the extension is simply continuous). It can be seen that any weak solution $u$ of \eqref{dir} satisfies $u\in C^\alpha_s(\overline\Omega)$ for some $\alpha\in(0,s)$, and if $u\neq 0$ then
\[u \in {\rm int}(C^0_s(\overline\Omega)_+)\]
(see \cite[Theorem 1.1]{IMS2} and \cite[Theorem 2.7]{IMP}). Obviously, also the proofs of Lemma \ref{fra} above and of some of the following results are simpler in such a case. On the other hand, in the singular case $p\in(1,2)$ this regularity theory is not available so far. Nevertheless, Lemmas \ref{smp} and \ref{fra} above hold the same and permit to face both the degenerate and the singular cases.
\end{remark}

\section{Uniqueness}\label{sec3}

In this section we prove that the solution of problem \eqref{dir}, if any, is unique. The argument is similar to that of \cite[Theorem 2.8]{IL} and makes a crucial use of Lemma \ref{fra} and the following discrete Picone's inequality from \cite[Proposition 2.2]{BS}:
\beq\label{dpi}
j_p(a-b)\Big(\frac{c^p}{a^{p-1}}-\frac{d^p}{b^{p-1}}\Big) \le |c-d|^p \ \text{for all $a,b\in\R^+_0$, $c,d\in\R^+$.}
\eeq
Our uniqueness result is the following:

\begin{proposition}\label{uni}
Let ${\bf H}$ hold. Then, \eqref{dir} has at most one solution.
\end{proposition}
\begin{proof}
Let $u,v\in\w\setminus\{0\}$ satisfy \eqref{wsl}. By Lemmas \ref{reg}, \ref{smp} we have $u,v\in C^\alpha(\overline\Omega)$, and we can find $C>1$ s.t.\ in $\Omega$
\[\frac{\ds}{C} \le u,\,v \le C\ds.\]
So, by Lemma \ref{fra} we have
\[\frac{u^p}{v^{p-1}},\,\frac{v^p}{u^{p-1}} \in \w.\]
Set $w=(u^p-v^p)^+$, then by the relations above
\[\frac{w}{v^{p-1}} = \Big(\frac{u^p}{v^{p-1}}-v\Big)^+ \in \w,\]
and similarly $w/u^{p-1}\in\w$. Testing \eqref{wsl} with such functions and recalling ${\bf H}$  (\ref{h3}), we have
\beq\label{uni1}
\Big\langle\fpl u,\frac{w}{u^{p-1}}\Big\rangle-\Big\langle\fpl v,\frac{w}{v^{p-1}}\Big\rangle = \int_{\{u>v\}}\Big(\frac{f(x,u)}{u^{p-1}}-\frac{f(x,v)}{v^{p-1}}\Big)(u^p-v^p)\,dx \le 0.
\eeq
To proceed, we prove that for all $x,y\in\R^N$
\beq\label{uni2}
j_p(u(x)-u(y))\Big(\frac{w(x)}{u(x)^{p-1}}-\frac{w(y)}{u(y)^{p-1}}\Big) \ge j_p(v(x)-v(y))\Big(\frac{w(x)}{v(x)^{p-1}}-\frac{w(y)}{v(y)^{p-1}}\Big)
\eeq
(with the usual convention that any function defined in $\Omega$ is identified with its $0$-extension to $\R^N$). Indeed, four cases may occur:
\begin{itemize}[leftmargin=1cm]
\item[$(a)$] if $u(x)>v(x)$, $u(y)>v(y)$, then by applying \eqref{dpi} twice we have
\begin{align*}
j_p(u(x)-u(y))\Big(\frac{w(x)}{u(x)^{p-1}}-\frac{w(y)}{u(y)^{p-1}}\Big) &= |u(x)-u(y)|^p-j_p(u(x)-u(y))\Big(\frac{v(x)^p}{u(x)^{p-1}}-\frac{v(y)^p}{u(y)^{p-1}}\Big) \\
&\ge |u(x)-u(y)|^p-|v(x)-v(y)|^p \\
&\ge j_p(v(x)-v(y))\Big(\frac{u(x)^p}{v(x)^{p-1}}-\frac{u(y)^p}{v(y)^{p-1}}\Big)-|v(x)-v(y)|^p \\
&= j_p(v(x)-v(y))\Big(\frac{w(x)}{v(x)^{p-1}}-\frac{w(y)}{v(y)^{p-1}}\Big);
\end{align*}
\item[$(b)$] if $u(x)>v(x)$, $u(y)\le v(y)$, then
\[\frac{u(y)}{u(x)} \le \frac{v(y)}{v(x)},\]
and since $j_p$ is increasing and $(p-1)$-homogeneous in $\R$, we have
\begin{align*}
j_p(u(x)-u(y))\Big(\frac{w(x)}{u(x)^{p-1}}-\frac{w(y)}{u(y)^{p-1}}\Big) &= j_p(u(x)-u(y))\frac{u(x)^p-v(x)^p}{u(x)^{p-1}} \\
&= j_p\Big(1-\frac{u(y)}{u(x)}\Big)(u(x)^p-v(x)^p) \\
&\ge  j_p\Big(1-\frac{v(y)}{v(x)}\Big)(u(x)^p-v(x)^p) \\
&= j_p(v(x)-v(y))\frac{u(x)^p-v(x)^p}{v(x)^{p-1}} \\
&= j_p(v(x)-v(y))\Big(\frac{w(x)}{v(x)^{p-1}}-\frac{w(y)}{v(y)^{p-1}}\Big);
\end{align*}
\item[$(c)$] if $u(x)\leq v(x)$, $u(y)> v(y)$, then
\[\frac{u(x)}{u(y)} \le \frac{v(x)}{v(y)},\]
and, similarly to the previous case, we have
\begin{align*}
j_p(u(x)-u(y))\Big(\frac{w(x)}{u(x)^{p-1}}-\frac{w(y)}{u(y)^{p-1}}\Big) &=- j_p(u(x)-u(y))\frac{u(y)^p-v(y)^p}{u(y)^{p-1}} \\
&\geq -j_p(v(x)-v(y))\frac{u(y)^p-v(y)^p}{v(y)^{p-1}} \\
&= j_p(v(x)-v(y))\Big(\frac{w(x)}{v(x)^{p-1}}-\frac{w(y)}{v(y)^{p-1}}\Big);
\end{align*}
\item[$(d)$] if $u(x)\le v(x)$, $u(y)\le v(y)$, then we have $w(x)=w(y)=0$ and \eqref{uni2} holds trivially.
\end{itemize}
Integrating \eqref{uni2} in $\R^N\times\R^N$ we have
\[\Big\langle\fpl u,\frac{w}{u^{p-1}}\Big\rangle \ge \Big\langle\fpl v,\frac{w}{v^{p-1}}\Big\rangle,\]
which along with \eqref{uni1} forces
\[\int_{\{u>v\}}\Big(\frac{f(x,u)}{u^{p-1}}-\frac{f(x,v)}{v^{p-1}}\Big)(u^p-v^p)\,dx = 0.\]
By ${\bf H}$  (\ref{h3}) (strict monotonicity), the integrand above is negative, so we deduce that $\{u>v\}$ has $0$-measure, i.e., $u\le v$ in $\Omega$. Similarly we see that $u\ge v$ in $\Omega$, and thus $u=v$.
\end{proof}

\section{Existence/I: necessary condition}\label{sec4}

In this section we assume the existence of a solution of \eqref{dir}, and we prove that
\[\lambda_1(a_0)<0<\lambda_1(a_\infty),\]
with $\lambda_1(a_0),\lambda_1(a_\infty)\in\R\cup\{\pm\infty\}$ defined by \eqref{la1}. We begin with the inequality on the right:

\begin{lemma}\label{inf}
Let ${\bf H}$ hold and $u$ be a solution of \eqref{dir}. Then, $\lambda_1(a_\infty)>0$.
\end{lemma}
\begin{proof}
From Lemmas \ref{reg}, \ref{smp} we know that $u\in C^\alpha(\overline\Omega)$ satisfies \eqref{wsl} and there exists $C>1$ s.t.\ in $\Omega$
\[\frac{\ds}{C} \le u \le C\ds.\]
In particular $u\in L^\infty(\Omega)$. Set for all $x\in\Omega$
\[a(x) = \frac{f(x,\|u\|_\infty)}{\|u\|_\infty^{p-1}}.\]
By ${\bf H}$  (\ref{h1})  (\ref{h3}) we have $a\in L^\infty(\Omega)$ and for a.e.\ $x\in\Omega$
\beq\label{inf1}
\frac{f(x,u)}{u^{p-1}} \ge a(x) > a_\infty(x),
\eeq
the first inequality being strict on a non-null subset of $\Omega$. Define $\lambda_1(a)\in\R$ according to \eqref{la1}, then the infimum is attained at some $v\in\w$ with $\|v\|_p=1$. Since $|v|\in\w$ with $\||v|\|\le\|v\|$, we may assume $v\ge 0$ in $\Omega$. In particular,
\[0 < \|v\|^p = \int_\Omega(\lambda_1(a)+a(x))v^p\,dx.\]
So $m=\lambda_1(a)+a\in L^\infty(\Omega)$ is a weight function s.t.\ $m^+\neq 0$. By \cite[Proposition 3.3]{I} and arguing as in Lemmas \ref{reg}, \ref{smp}, we see that $v\in C^\alpha(\overline\Omega)_+$ is unique and, for a possibly bigger $C>1$, we have in $\Omega$
\[\frac{\ds}{C} \le v \le C\ds.\]
Equivalently, $v$ is the unique positive, $L^p(\Omega)$-normalized principal eigenfunction of \eqref{evp}. Thanks to the estimates above, we can find $\tau>0$ s.t.\ $u<\tau v$ in $\Omega$, hence in particular $\|u\|_p<\tau$. Then, $\tau v\in\w_+\cap C^\alpha(\overline\Omega)\setminus\{0\}$ satisfies weakly in $\Omega$
\beq\label{inf2}
\fpl(\tau v) = (\lambda_1(a)+a(x))(\tau v)^{p-1}.
\eeq
By Lemma \ref{fra} we have
\[\frac{u^p}{(\tau v)^{p-1}},\,\frac{(\tau v)^p}{u^{p-1}} \in \w.\]
Testing \eqref{wsl} and \eqref{inf2} with convenient functions, and using \eqref{inf1} and the normalization of $v$, we have
\begin{align*}
&\Big\langle\fpl u,u-\frac{(\tau v)^p}{u^{p-1}}\Big\rangle+\Big\langle\fpl(\tau v),\tau v-\frac{u^p}{(\tau v)^{p-1}}\Big\rangle \\
&= \int_\Omega f(x,u)\frac{u^p-(\tau v)^p}{u^{p-1}}\,dx+\int_\Omega(\lambda_1(a)+a(x))(\tau v)^{p-1}\frac{(\tau v)^p-u^p}{(\tau v)^{p-1}}\,dx \\
&= \int_\Omega\Big(\frac{f(x,u)}{u^{p-1}}-a(x)\Big)(u^p-(\tau v)^p)\,dx+\lambda_1(a)\int_\Omega((\tau v)^p-u^p)\,dx \\
&< \lambda_1(a)(\tau^p-\|u\|_p^p).
\end{align*}
Besides, by the aforementioned Picone's inequality \eqref{dpi} we have
\begin{align*}
&\Big\langle\fpl u,u-\frac{(\tau v)^p}{u^{p-1}}\Big\rangle+\Big\langle\fpl(\tau v),\tau v-\frac{u^p}{(\tau v)^{p-1}}\Big\rangle \\
&= \|u\|^p+\|\tau v\|^p \\
&- \iint_{\R^N\times\R^N}j_p(u(x)-u(y))\Big(\frac{(\tau v(x))^p}{u(x)^{p-1}}-\frac{(\tau v(y))^p}{u(y)^{p-1}}\Big)\,\frac{dx\,dy}{|x-y|^{N+ps}} \\
&- \iint_{\R^N\times\R^N}j_p(\tau v(x)-\tau v(y))\Big(\frac{u(x)^p}{(\tau v(x))^{p-1}}-\frac{u(y)^p}{(\tau v(y))^{p-1}}\Big)\,\frac{dx\,dy}{|x-y|^{N+ps}} \ge 0.
\end{align*}
Concatenating the inequalities above, we get
\[\lambda_1(a)(\tau^p-\|u\|_p^p) > 0,\]
which along with $\tau>\|u\|_p$ implies $\lambda_1(a)>0$. Finally, we note that by \eqref{inf1} we have for all $w\in\w$ with $\|w\|_p=1$
\[\|w\|^p-\int_{\{w\neq 0\}}a_\infty(x)|w|^p\,dx \ge \|w\|^p-\int_{\{w\neq 0\}}a(x)|w|^p\,dx \ge \lambda_1(a).\]
Taking the infimum over $w$, we find
\[\lambda_1(a_\infty) \ge \lambda_1(a) > 0\]
and conclude.
\end{proof}

\begin{remark}\label{bda0}
In the proof of Lemma \ref{inf} we have seen, {\em en passant}, that the mapping $a\mapsto\lambda_1(a)$ defined in \eqref{la1} is monotone nonincreasing with respect to the pointwise ordering in $\Omega$. If $a_\infty\in L^\infty(\Omega)$ (for instance when $f$ satisfies a global growth condition), we would get  $\lambda_1(a_\infty)>\lambda_1(a)$ reasoning as in \cite[Proposition 4.2]{I}.
\end{remark}

The argument for the inequality on the left is simpler:

\begin{lemma}\label{zer}
Let ${\bf H}$ hold and $u$ be a solution of \eqref{dir}. Then, $\lambda_1(a_0)<0$.
\end{lemma}
\begin{proof}
If $a_0=\infty$ on a non-null subset of $\Omega$, then by \eqref{la1} we have $\lambda_1(a_0)=-\infty$ and there is nothing to prove. So we may assume $a_0<\infty$ in $\Omega$. By ${\bf H}$  (\ref{h3}) and \eqref{wmp1} (with $\delta=1$) we have for a.e.\ $x\in\Omega$
\[a_0(x) > f(x,1) \ge -C_1.\]
As in Lemma \ref{inf} we have $u\in C^\alpha(\overline\Omega)$ and $u>0$ in $\Omega$. Testing \eqref{wsl} with $u\in\w$ and using ${\bf H}$  (\ref{h3}) we have
\[\|u\|^p = \langle\fpl u,u\rangle = \int_\Omega f(x,u)u\,dx < \int_\Omega a_0(x)u^p\,dx.\]
So by \eqref{la1} we have
\[\lambda_1(a_0) \le \frac{\|u\|^p-\int_\Omega a_0(x)u^p\,dx}{\|u\|_p^p} < 0\]
and we conclude.
\end{proof}

Summarizing Lemmas \ref{inf}, \ref{zer}, we have the following necessary condition for existence:

\begin{proposition}\label{nec}
Let ${\bf H}$ hold and \eqref{dir} have a solution. Then, $\lambda_1(a_0)<0<\lambda_1(a_\infty)$.
\end{proposition}

\section{Existence/II: sufficient condition}\label{sec5}

The most delicate part of this study consists in proving that $\lambda_1(a_0)<0<\lambda_1(a_\infty)$ implies existence of a solution to \eqref{dir}. Following \cite{BO,DS}, we use a variational approach. Recalling the extended definition of $f$, set for all $(x,t)\in\Omega\times\R$
\[F(x,t) = \int_0^t f(x,\tau)\,d\tau,\]
and for all $u\in\w$
\[\Phi(u) = \frac{\|u\|^p}{p}-\int_\Omega F(x,u)\,dx.\]
Under hypotheses ${\bf H}$, and especially due to the lack of a growth condition on $f(x,\cdot)$ from below, we cannot expect $\Phi$ to be G\^ateaux differentiable, in fact $\Phi$ is not even continuous in $\w$. So, in the following lemmas we will explore the properties of $\Phi$:

\begin{lemma}\label{lsc}
Let ${\bf H}$ hold. Then, $\Phi:\w\to\R\cup\{\infty\}$ is sequentially weakly l.s.c.
\end{lemma}
\begin{proof}
By ${\bf H}$  (\ref{h2}) we have for a.e.\ $x\in\Omega$ and all $t\in\R^+$
\[F(x,t) \le \int_0^t c_0(1+\tau^{p-1})\,d\tau = c_0\Big(t+\frac{t^p}{p}\Big),\]
while for all $t\in\R^-$ we have by ${\bf H}$  (\ref{h1}) and \eqref{wmp2}
\[F(x,t) = \int_0^t f(x,0)\,d\tau \le 0.\]
The estimates above imply for a.e.\ $x\in\Omega$ and all $t\in\R$
\beq\label{lsc1}
F(x,t) \le C(1+|t|^p).
\eeq
By \eqref{lsc1} and the continuous embedding $\w\hookrightarrow L^p(\Omega)$, we have $\Phi(u)>-\infty$ for all $u\in\w$. Now let $(u_n)$ be a sequence in $\w$ s.t.\ $u_n\rightharpoonup u$ in $\w$. Then
\[\liminf_n \frac{\|u_n\|^p}{p} \ge \frac{\|u\|^p}{p}.\]
Passing if necessary to a subsequence, we may assume that $u_n\to u$ in $L^p(\Omega)$ (by the compact embedding $\w\hookrightarrow L^p(\Omega)$), that $u_n(x)\to u(x)$ for a.e.\ $x\in\Omega$ and there is $g\in L^p(\Omega)$ s.t.\ $|u_n|\le g$ in $\Omega$ for all $n\in\N$. By continuity of the Nemytzskii operator, we have for a.e.\ $x\in\Omega$
\[\lim_n F(x,u_n(x)) = F(x,u(x)).\]
Besides, by \eqref{lsc1} we have for all $n\in\N$ and a.e.\ $x\in\Omega$
\[F(x,u_n(x)) \le C(1+g(x)^p).\]
So we can apply Fatou's lemma and find
\[\limsup_{n} \int_\Omega F(x,u_n)\,dx \le \int_\Omega F(x,u)\,dx.\]
Thus,
\[\liminf_{n} \Phi(u_n) \ge \Phi(u),\]
which proves that $\Phi$ is sequentially weakly l.s.c.
\end{proof}

The behavior of $\Phi$ at infinity is governed by $a_\infty$:

\begin{lemma}\label{coe}
Let ${\bf H}$ hold and $\lambda_1(a_\infty)>0$. Then, $\Phi$ is coercive in $\w$.
\end{lemma}
\begin{proof}
Using ${\bf H}$  (\ref{h3}) and de l'H\^opital's rule, we see that for a.e.\ $x\in\Omega$
\beq\label{coe1}
\lim_{t\to+\infty}\frac{F(x,t)}{t^p} = \frac{a_\infty(x)}{p}.
\eeq
We aim at proving that
\[\lim_{\|u\|\to\infty}\Phi(u) = +\infty.\]
Arguing by contradiction, let $(u_n)$ be a sequence in $\w$ s.t.\ $\|u_n\|\to\infty$ and $\Phi(u_n)\le C$ for some $C\in \R$ and for all $n\in\N$. Then, by \eqref{lsc1} we have for all $n\in\N$
\[\frac{\|u_n\|^p}{p} \le C+\int_\Omega F(x,u_n)\,dx \le C(1+\|u_n\|_p^p).\]
So, we have that $\|u_n\|_p\to\infty$, as well. Now, for all $n\in\N$ set
\[\rho_n=\|u_n\|_p, \ v_n=\frac{u_n}{\rho_n}\in\w.\]
Then clearly $\rho_n\to\infty$ and $\|v_n\|_p=1$ for all $n\in\N$. As above we have for all $n\in\N$
\[\frac{\|v_n\|^p}{p} \le C\frac{1+\rho_n^p}{\rho_n^p} \le C,\]
so $(v_n)$ is bounded in $\w$. By the reflexivity of $\w$ and the compact embedding $\w\hookrightarrow L^p(\Omega)$, possibly passing to a subsequence, we have that $v_n\rightharpoonup v$ in $\w$ and $v_n\to v$ in $L^p(\Omega)$. Hence, we have in particular
\beq\label{coe2}
\|v\|_p = 1.
\eeq
Passing if necessary to a further subsequence, we have $v_n(x)\to v(x)$ for a.e.\ $x\in\Omega$, with dominated convergence in $L^p(\Omega)$. We claim that
\beq\label{coe3}
\limsup_n \int_\Omega\frac{F(x,\rho_nv_n)}{\rho_n^p}\,dx \le \int_{\{v>0\}}\frac{a_\infty(x)v^p}{p}\,dx.
\eeq
Indeed, for all $n\in\N$ we have
\begin{align*}
\int_\Omega\frac{F(x,\rho_nv_n)}{\rho_n^p}\,dx &= \int_{\{v>0\}}\frac{F(x,\rho_nv^+_n)}{\rho_n^p}\,dx+\int_{\{v\le 0\}}\frac{F(x,\rho_nv_n^+)}{\rho_n^p}\,dx+\int_{\{v_n\le 0\}}\frac{F(x,\rho_nv_n)}{\rho_n^p}\,dx \\
&= I^1_n+I^2_n+I^3_n.
\end{align*}
We consider the three integrals separately:
\begin{itemize}[leftmargin=1cm]
\item[$(a)$] In the set $\{v>0\}$ we have a.e.\ $v_n>0$ for all $n\in\N$ big enough, so by \eqref{coe1} and $\rho_n\to\infty$ we get
\[\lim_n\frac{F(x,\rho_nv_n)}{\rho_n^p} = \lim_n\frac{F(x,\rho_nv_n)}{(\rho_nv_n)^p}\,v_n^p = \frac{a_\infty(x)v^p}{p}.\]
Thus, by applying Fatou's lemma, as in Lemma \ref{lsc}, we have
\[\limsup_n I^1_n \le  \int_{\{v>0\}}\frac{a_\infty(x)v^p}{p}\,dx.\]
\item[$(b)$] In $\{v\le 0\}$ we have $v_n^+\to 0$, with dominated convergence in $L^p(\Omega)$, so
\[\lim_n\int_{\{v\le 0\}}(v_n^+)^p\,dx = 0,\]
which in turn implies, along with \eqref{lsc1},
\begin{align*}
\limsup_n I^2_n &\le C\limsup_n\int_{\{v\le 0\}}\frac{1+(\rho_nv_n^+)^p}{\rho_n^p}\,dx \\
&\le C\lim_n\Big(\frac{1}{\rho_n^p}+\int_{\{v\le 0\}}(v_n^+)^p\,dx\Big) = 0.
\end{align*}
\item[$(c)$] Finally,  in $\{v_n\le 0\}$ for all $n\in\N$ we have
\[F(x,\rho_nv_n) = f(x,0)\rho_nv_n,\]
so by \eqref{wmp2}
\[\limsup_n I^3_n \le \limsup_n\int_{\{v_n\le 0\}}\frac{f(x,0)v_n}{\rho_n^{p-1}}\,dx \le 0.\]
\end{itemize}
Adding up the relations above, we find \eqref{coe3}. Now, by $v_n\rightharpoonup v$ in $\w$ and \eqref{coe3} we have
\begin{align}\label{coe4}
\frac{\|v\|^p}{p} &\le \liminf_n\frac{\|v_n\|^p}{p} \\
\nonumber &\le \limsup_n \frac{1}{\rho_n^p}\Big(\Phi(u_n)+\int_\Omega F(x,u_n)\,dx\Big) \\
\nonumber &\le \limsup_n\Big(\frac{C}{\rho_n^p}+\int_\Omega \frac{F(x,\rho_nv_n)}{\rho_n^p}\,dx\Big) \\
\nonumber &\le \int_{\{v>0\}}\frac{a_\infty(x)v^p}{p}\,dx.
\end{align}
Recalling \eqref{la1} (with $a=a_\infty$), by the inequality above and by the fact that $\|v^+\|\leq \|v\|$, we have
\[\lambda_1(a_\infty)\|v^+\|_p^p \le \|v^+\|^p-\int_{\{v>0\}}a_\infty(x)v^p\,dx \le 0.\]
Since by assumption $\lambda_1(a_\infty)>0$, we have $v\le 0$ in $\Omega$. But then again, by \eqref{coe4}, we have 
\[\|v\|^p \le \int_{\{v>0\}}a_\infty(x)v^p\,dx = 0.\]
So $v=0$, against \eqref{coe2}. This contradiction proves that $\Phi$ is coercive in $\w$.
\end{proof}

On the other hand, $a_0$ determines the behavior of $\Phi$ near the origin:

\begin{lemma}\label{min}
Let ${\bf H}$ hold and $\lambda_1(a_0)<0$. Then, there exists $\bar u\in\w$ s.t.\ $\Phi(\bar u)<0$.
\end{lemma}
\begin{proof}
Using ${\bf H}$  (\ref{h3}) and de l'H\^opital's rule, we see that for a.e.\ $x\in\Omega$
\beq\label{min1}
\lim_{t\to 0^+}\frac{F(x,t)}{t^p} = \frac{a_0(x)}{p}.
\eeq
By \eqref{la1}, we can find $v\in\w$ s.t.\
\beq\label{min2}
\|v\|^p-\int_{\{v\neq 0\}}a_0(x)|v|^p\,dx < 0.
\eeq
By a density argument, and replacing if necessary $v$ with $|v|$, we may assume $v\in L^\infty(\Omega)_+$. By \eqref{min1}, in $\{v>0\}$ we have for a.e.\ $x\in\Omega$
\[\lim_{\eps\to 0^+}\frac{F(x,\eps v)}{\eps^p} = \frac{a_0(x)v^p}{p}.\]
Besides, given $\delta>\|v\|_\infty$, by \eqref{wmp1}, we have in $\{v>0\}$ and for all $\eps\in(0,1)$
\[\frac{F(x,\eps v)}{\eps^p} = \int_0^{\eps v}\frac{f(x,t)}{\eps^p}\,dt \ge -\frac{C_\delta\|v\|_\infty^p}{p}.\]
So we can apply Fatou's lemma and \eqref{min2} and find
\[\liminf_{\eps\to 0^+}\int_{\{v>0\}}\frac{F(x,\eps v)}{\eps^p}\,dx \ge \int_{\{v>0\}}\frac{a_0(x)v^p}{p}\,dx > \frac{\|v\|^p}{p}.\]
Then, for all $\eps>0$ small enough we have
\[\int_{\{v>0\}}\frac{F(x,\eps v)}{\eps^p}\,dx > \frac{\|v\|^p}{p}.\]
Now we set $\bar u=\eps v\in\w$ and compute
\begin{align*}
\Phi(\bar u) &= \frac{\eps^p\|v\|^p}{p}-\int_{\{v>0\}} F(x,\eps v)\,dx \\
&= \eps^p\Big(\frac{\|v\|^p}{p}-\int_{\{v>0\}}\frac{F(x,\eps v)}{\eps^p}\,dx\Big) < 0,
\end{align*}
thus concluding.
\end{proof}

\begin{remark}\label{str}
In Lemmas \ref{lsc}, \ref{coe}, and \ref{min} above we did not use the {\it strict} monotonicity of ${\bf H}$  (\ref{h3}). In fact, all the results in this Section can be proved, with minor adjustments, under the weaker condition \eqref{wmp1} in place of ${\bf H}$  (\ref{h3}).
\end{remark}

The final step consists in proving the sufficient condition for existence:

\begin{proposition}\label{suf}
Let ${\bf H}$ hold and $\lambda_1(a_0)<0<\lambda_1(a_\infty)$. Then, \eqref{dir} has a solution.
\end{proposition}
\begin{proof}
From Lemmas \ref{lsc}, \ref{coe} we know that $\Phi:\w\to\R\cup\{\infty\}$ is sequentially weakly l.s.c.\ and coercive. By reflexivity of $\w$, there exists $u\in\w$ s.t.\
\beq\label{suf1}
\Phi(u) = \inf_{v\in\w}\Phi(v).
\eeq
We may assume $u\ge 0$ in $\Omega$. Otherwise, we replace $u$ with $u^+\in\w$. Indeed, by \eqref{wmp2} we have
\begin{align*}
\Phi(u^+) &= \frac{\|u^+\|^p}{p}-\int_{\{u>0\}}F(x,u)\,dx \\
&\le \frac{\|u\|^p}{p}-\int_\Omega F(x,u)\,dx+\int_{\{u\le 0\}}f(x,0)u\,dx \le \Phi(u).
\end{align*}
Also, by Lemma \ref{min} we have
\[\Phi(u) \le \Phi(\bar u) < 0,\]
hence $u\in\w_+\setminus\{0\}$. Due to the lack of differentiability of $\Phi$, we cannot infer immediately that $u$ satisfies \eqref{wsl} and afterwards apply the regularity theory developed in Section \ref{sec2}. Instead, we prove independently that in \eqref{suf1} we may assume
\beq\label{suf2}
u\in L^\infty(\Omega).
\eeq
To this end, following \cite{BO}, we introduce a sequence of truncated reactions by setting for all $k\in\N$ and all $(x,t)\in\Omega\times\R$
\[f_k(x,t) = f(x,t^+)\vee(-k(t^+)^{p-1}).\]
By ${\bf H}$, it is immediately seen that $f_k:\Omega\times\R\to\R$ is a Carath\'eodory function for every $k\in \N$. By ${\bf H}$  (\ref{h1}) we have for all $t\in\R$
\[f_k(\cdot,t) \in L^\infty(\Omega).\]
Plus, by ${\bf H}$  (\ref{h2}) we have for a.e.\ $x\in\Omega$ and all $t\in\R^+$
\[-kt^{p-1} \le f_k(x,t) \le (c_0\vee k)(1+t^{p-1}),\]
while by \eqref{wmp2} we have for a.e.\ $x\in\Omega$ and all $t\in\R^-$
\[|f_k(x,t)| \le \|f(\cdot,0)\|_\infty.\]
In conclusion, for any $k\in\N$ we can find $c_k>0$ s.t.\ for a.e.\ $x\in\Omega$ and all $t\in\R$
\beq\label{suf3}
|f_k(x,t)| \le c_k(1+|t|^{p-1}).
\eeq
By ${\bf H}$  (\ref{h3}) we have for a.e.\ $x\in\Omega$ and all $0<t<t'$
\[\frac{f_k(x,t)}{t^{p-1}} = \frac{f(x,t)}{t^{p-1}}\vee(-k) \ge \frac{f(x,t')}{(t')^{p-1}}\vee(-k) = \frac{f_k(x,t')}{(t')^{p-1}},\]
i.e., the map
\[t\mapsto\frac{f_k(x,t)}{t^{p-1}}\]
is nonincreasing in $\R^+_0$, as well. Thus, $f_k$ satisfies ${\bf H}$ (but the strict monotonicity in ${\bf H}$ (\ref{h3})), and in addition the bilateral growth condition \eqref{suf3}. In addition, we have for all $k\in\N$, a.e.\ $x\in\Omega$, and all $t\in\R$ the following useful inequality:
\beq\label{suf4}
f_k(x,t) \ge f_{k+1}(x,t) \ge f(x,t).
\eeq
By monotonicity, for all $k\in \N$ we may define two measurable functions by setting for a.e.\ $x\in\Omega$ 
\[a_0^k(x) = \lim_{t\to 0^+}\frac{f_k(x,t)}{t^{p-1}}, \ a_\infty^k(x) = \lim_{t\to\infty}\frac{f_k(x,t)}{t^{p-1}}.\]
Some remarks on the sequences $(a_0^k)$, $(a_\infty^k)$ are now in order. First we focus on $(a_0^k)$. From \eqref{suf4} we have for a.e.\ $x\in\Omega$ and all $t\in\R^+_0$
\[\frac{f_k(x,t)}{t^{p-1}} \ge \frac{f(x,t)}{t^{p-1}}.\]
Passing to the limit as $t\to 0^+$ gives
\[a_0^k(x) \ge a_0(x),\]
hence in particular $a^k_0$ is bounded from below in $\Omega$. Now define $\lambda_1(a^k_0)$ as in \eqref{la1}. We have already seen (in Remark \ref{bda0}) that the map $a\mapsto\lambda_1(a)$ is nonincreasing, so by the main assumption we have for all $k\in\N$
\beq\label{suf5}
\lambda_1(a_0^k) \le \lambda_1(a_0) < 0.
\eeq
The case for $(a_\infty^k)$ is subtler. First note that, by \eqref{suf3}, we have $a_\infty^k\in L^\infty(\Omega)$ for all $k\in\N$. Also, dividing \eqref{suf4} by $t^{p-1}$ and then letting $t\to\infty$, we get for all $k\in\N$ and a.e.\ $x\in\Omega$
\beq\label{suf6}
a_\infty^k(x) \ge a_\infty^{k+1}(x) \ge a_\infty(x).
\eeq
In fact we have for a.e.\ $x\in\Omega$
\beq\label{suf7}
\lim_k a_\infty^k(x) = a_\infty(x).
\eeq
Indeed, by \eqref{suf6} the sequence $(a_\infty^k(x))$ is nonincreasing, hence the limit above exists, and in addition
\[\lim_k a_\infty^k(x) \ge a_\infty(x).\]
Now fix $M>a_\infty(x)$. We can find $T>0$ (depending on $x$) s.t.\ for all $t\ge T$
\[\frac{f(x,t)}{t^{p-1}} < M.\]
For any such $t$, choose $k\in\N$ s.t.\
\[k \ge \frac{\|f(\cdot,t)\|_\infty}{t^{p-1}},\]
so we have
\[\frac{f_k(x,t)}{t^{p-1}} = \frac{f(x,t)}{t^{p-1}} < M.\]
Finally, let $t\to\infty$ to get
\[a_\infty^k(x) \le M,\]
which completes the proof of \eqref{suf7}. The next step consists in proving that
\beq\label{suf8}
\lim_k \lambda_1(a_\infty^k) = \lambda_1(a_\infty) > 0.
\eeq
By \eqref{suf6} and the  monotonicity of $a\mapsto\lambda_1(a)$ (Remark \ref{bda0} again), we see that $(\lambda_1(a_\infty^k))$ is a nondecreasing sequence, hence the limit above exists. Still by \eqref{suf6} we have
\[\lim_k\lambda_1(a_\infty^k) \le \lambda_1(a_\infty).\]
Now, if $\lambda_1(a_\infty^k)\to\infty$, then $\lambda_1(a_\infty)=\infty$ and there is nothing to prove. So, let us assume that $(\lambda_1(a_\infty^k))$ is bounded from above. For all $k\in\N$ there exists $v_k\in\w$ s.t.\ $\|v_k\|_p=1$ and
\[\|v_k\|^p-\int_{\{v_k\neq 0\}}a_\infty^k(x)|v_k|^p\,dx < \lambda_1(a_\infty^k)+\frac{1}{k}.\]
By \eqref{suf3} and \eqref{suf6}, we can find $M>0$ s.t.\ $a_\infty^k\le M$ in $\Omega$, for all $k\in\N$. So the last inequality gives for all $k\in\N$
\begin{align*}
\|v_k\|^p &< \int_{\{v_k\neq 0\}}a_\infty^k(x)|v_k|^p\,dx+\lambda_1(a_\infty^k)+\frac{1}{k} \\
&\le M\|v_k\|_p^p+C \le C.
\end{align*}
Therefore, $(v_k)$ is bounded in $\w$. By reflexivity and the compact embedding $\w\hookrightarrow L^p(\Omega)$, passing to a subsequence if necessary, we have $v_k\rightharpoonup v$ in $\w$ and $v_k\to v$ in $L^p(\Omega)$. Hence in particular $\|v\|_p=1$. By weak convergence we have
\[\liminf_k\|v_k\|^p \ge \|v\|^p.\]
Also, by \eqref{suf7}, $a_\infty^k(x)|v_k(x)|^p\to a_\infty(x)|v(x)|^p$ for a.e.\ $x\in\Omega$ with dominated convergence from above, so by Fatou's lemma
\[\limsup_k\int_{\{v_k\neq 0\}} a_\infty^k(x)|v_k|^p\,dx \le \int_{\{v\neq 0\}}a_\infty(x)|v|^p\,dx.\]
Therefore, using also \eqref{la1}, we have
\begin{align*}
\lambda_1(a_\infty) &\le \|v\|^p-\int_{\{v\neq 0\}}a_\infty(x)|v|^p\,dx \\
&\le \liminf_k\Big(\|v_k\|^p-\int_{\{v_k\neq 0\}}a_\infty^k(x)|v_k|^p\,dx\Big) \\
&\le \lim_k\Big(\lambda_1(a_\infty^k)+\frac{1}{k}\Big) = \lim_k\lambda_1(a_\infty^k),
\end{align*}
hence we have \eqref{suf8}.
\vskip2pt
\noindent
Now, by \eqref{suf5} and \eqref{suf8} we can fix $k\in\N$ s.t.\
\[\lambda_1(a_0^k) < 0 < \lambda_1(a_\infty^k).\]
Set for all $(x,t)\in\Omega\times\R$
\[F_k(x,t) = \int_0^t f_k(x,\tau)\,d\tau,\]
and for all $v\in\w$
\[\Phi_k(v) = \frac{\|v\|^p}{p}-\int_\Omega F_k(x,v)\,dx.\]
Arguing as above and applying Lemmas \ref{lsc}, \ref{coe}, and \ref{min} (recalling Remark \ref{str}), we find $u_k\in\w_+\setminus\{0\}$ s.t.\
\beq\label{suf9}
\Phi_k(u_k) = \inf_{v\in\w}\Phi_k(v).
\eeq
But now, by \eqref{suf3} we have $\Phi_k\in C^1(\w)$ with derivative given for all $v,\varphi\in\w$ by
\[\langle\Phi'_k(v),\varphi\rangle = \langle\fpl v,\varphi\rangle-\int_\Omega f_k(x,v)\varphi\,dx.\]
So we can differentiate in \eqref{suf9} and see that $u_k$ is a weak solution (in the sense of Section \ref{sec2}) of the following problem:
\[\begin{cases}
\fpl u_k = f_k(x,u_k) & \text{in $\Omega,$}\\
u_k = 0 & \text{in $\Omega^c$.}
\end{cases}\]
Reasoning as in Lemmas \ref{reg} and \ref{smp}, we see that $u_k\in C^\alpha(\overline\Omega)$ and there exists $C>1$ s.t.\ in $\Omega$
\[\frac{\ds}{C} \le u_k \le C\ds.\]
Define
\[\underline{u}_k = u\wedge u_k, \ \overline{u}_k = u\vee u_k.\]
By the lattice structure of $\w$ we have $\underline{u}_k,\overline{u}_k\in\w$. Also, since $0\le\underline{u}_k\le u_k$ in $\Omega$, we clearly have $\underline{u}_k\in L^\infty(\Omega)$. Now we claim that
\beq\label{suf10}
\Phi(\underline{u}_k) \le \Phi(u).
\eeq
Indeed, by \eqref{suf9} we have
\[\Phi_k(u_k) \le \Phi_k(\overline{u}_k).\]
By the inequality above and \eqref{suf4} we have
\begin{align*}
\frac{\|u_k\|^p}{p}-\frac{\|\overline{u}_k\|^p}{p} &\le \int_\Omega F_k(x,u_k)\,dx-\int_\Omega F_k(x,\overline{u}_k)\,dx \\
&= \int_{\{u>u_k\}}\big(F_k(x,u_k)-F_k(x,u)\big)\,dx \\
&= \int_{\{u>u_k\}}\int_u^{u_k} f_k(x,t)\,dt\,dx \\
&\le \int_{\{u>u_k\}}\int_u^{u_k} f(x,t)\,dt\,dx \\
&= \int_{\{u>u_k\}}\big(F(x,u_k)-F(x,u)\big)\,dx \\
&= \int_\Omega F(x,\underline{u}_k)\,dx-\int_\Omega F(x,u)\,dx.
\end{align*}
Besides, by the submodularity inequality \eqref{sub} we have
\[\frac{\|\underline{u}_k\|^p}{p}+\frac{\|\overline{u}_k\|^p}{p} \le \frac{\|u\|^p}{p}+\frac{\|u_k\|^p}{p}.\]
Concatenating the last relations we have
\[\frac{\|\underline{u}_k\|^p}{p}-\frac{\|u\|^p}{p} \le \int_\Omega F(x,\underline{u}_k)\,dx-\int_\Omega F(x,u)\,dx,\]
which is equivalent to \eqref{suf10}. Thus, replacing if necessary $u$ with $\underline{u}_k$ in \eqref{suf1}, we finally get \eqref{suf2}. By \eqref{wmp1} with $\delta>\|u\|_\infty$ and ${\bf H}$  (\ref{h2}), we have for a.e.\ $x\in\Omega$ and all $t\in[0,\delta]$
\[|f(x,t)| \le C(1+t^{p-1}).\]
Such local bilateral growth condition allows us to differentiate in \eqref{suf1}, thus finding that $u\in\w\setminus\{0\}$ satisfies \eqref{wsl}. Then we can apply Lemmas \ref{reg}, \ref{smp} and conclude that $u\in C^\alpha(\overline\Omega)$ and $u>0$ in $\Omega$, hence $u$ solves \eqref{dir}.
\end{proof}

\noindent
{\bf Conclusion.} Simply lining up Propositions \ref{uni}, \ref{nec}, and \ref{suf} we have the complete proof of Theorem \ref{main}.

\section*{Acknowledgements}
\noindent
Both authors are members of GNAMPA (Gruppo Nazionale per l'Analisi Matematica, la Probabilit\`a e le loro Applicazioni) of INdAM (Istituto Nazionale di Alta Matematica 'Francesco Severi'). The first author is supported by the research project {\it Analysis of PDE's in connection with real phenomena} (CUP F73C22001130007, Fondazione di Sardegna 2021). The second author is supported by the FFABR 'Fondo per il finanziamento delle attivit\`a base di ricerca' 2017 and INdAM-GNAMPA Project  2022 {\it PDE ellittiche a diffusione mista}. The authors wish to thank S.\ Mosconi for drawing their attention to the precious submodularity inequality \eqref{sub}.

\end{document}